\documentclass[11pt]{article}
\usepackage{amsmath}
\usepackage{amssymb}
\usepackage{amsthm}
\usepackage{multirow}
\usepackage{color}
\usepackage{longtable}
\usepackage{array}
\usepackage{url}
\usepackage{enumerate}

\newtheorem{theorem}{Theorem}[section]
\newtheorem{definition}{Definition}[section]
\newtheorem{lemma}[theorem]{Lemma}
\newtheorem{example}[theorem]{Example}

\newtheorem{remark}{Remark}[section]

\title{Some New Balanced and Almost Balanced Quaternary Sequences with Low Autocorrelation}

\author{
Jerod Michel\thanks{Corresponding author. Email Address: michelj@sustc.edu.cn.}
\thanks{J. Michel is with the Department of Computer Science and Engineering, Southern University of Science and Technology, Shenzhen 518055, China.},
Qi Wang\thanks{Q. Wang is with the Department of Computer Science and Engineering, Southern University of Science and Technology, Shenzhen 518055, China.
\newline
The authors were supported in part by the Ministry of Science and Technology (MOST) of China under the Grant No. 2017YFC0804002, the Shenzhen fundamental research programs under Grant No. JCYJ20150630145302234, and the National Natural Science Foundation of China under Grant No. 11601220 and Grant No. 61672015.
} \\
}

\begin{document}
\date{}\maketitle

\begin{abstract}
Quaternary sequences of both even and odd period having low autocorrelation are studied. We construct new families of balanced quaternary sequences of odd period and low autocorrelation using cyclotomic classes of order eight, as well as investigate the linear complexity of some known quaternary sequences of odd period. We discuss a construction given by Chung et al. in ``New Quaternary Sequences with Even Period and Three-Valued Autocorrelation'' [IEICE Trans. Fundamentals Vol. E93-A, No. 1 (2010)] first by pointing out a slight modification (thereby obtaining new families of balanced and almost balanced quaternary sequences of even period and low autocorrelation), then by showing that, in certain cases, this slight modification greatly simplifies the construction given by Shen et al. in ``New Families of Balanced Quaternary Sequences of Even Period with Three-level Optimal Autocorrelation'' [IEEE Comm. Letters DOI10.1109/LCOMM.2017.26611750 (2017)]. We investigate the linear complexity of these sequences as well.

\end{abstract}


\section{Introduction}\label{sec1}
The periodic autocorrelation, balancedness, and linear complexity are all measures of interest when designing sequences. It is desirable to have small autocorrelation, large linear complexity, and balancedness for applications in certain communications systems, cryptography and digital systems (see \cite{CDR}, \cite{FD}, \cite{GOL}, \cite{MICH} and \cite{SOSL}).
\par
Binary sequences with low autocorrelation are important building blocks for constructing quaternary sequences with low autocorrelation (\cite{CAI}, \cite{GOL}, \cite{SU}, \cite{TD}). The use of cyclotomic classes and generalized cyclotomic classes to construct binary sequences is a well-known and extensively used method \cite{CUN}, \cite{DHM}. Sidelnikov, in \cite{SID}, (and later, Lempel, Cohn and Eastman in \cite{LEM}) used quadratic residues modulo some prime power $q$ to construct binary sequences of period $q-1$ with optimal autocorrelation. No et al., in \cite{NO1}, found that, by altering these sequences by a single bit yields even more binary sequences of the same period with optimal autocorrelation. In \cite{ACH}, Arasu et al. constructed many classes of binary sequences with optimal autocorrelation of period $4l$ by interleaving four binary sequences of period $l$ having ideal autocorrelation. This work was further generalized in \cite{TD} by Tang and Ding.
\par
\par
There has been some recent progress in designing quaternary sequences as well. There seems to be much less in the literature on quaternary sequences of odd period with low autocorrelation. Quaternary sequences of odd prime period were constructed by Sidelnikov in \cite{SID}, by Green and Green in \cite{GREEN1}, by Tang and Lindner in \cite{TANG1}, and by Yang and Ke in \cite{YANG1}. Quaternary sequences of odd composite ($pq$ where $q-p=2$ or $4$) period were constructed by Green and Green in \cite{GREEN}, and by Han and Yang in \cite{HAN} and \cite{YANG}. Recent work on quaternary sequences with even period and low autocorrelation is more abundant. In \cite{TD}, Tang and Ding constructed several families of balanced and almost balanced quaternary sequences of period $4l$ where $l\equiv 3 \ ($mod $4)$. In \cite{KIM}, Kim et al. constructed families of balanced quaternary sequences of even period with optimal autocorrelation, in \cite{KIM1}, the same authors constructed quaternary sequences with ideal autocorrelation from Legendre sequences, and in \cite{JANG}, the same authors constructed new quaternary sequences with ideal autocorrelation from binary sequences with ideal autocorrelation. Autocorrelation of certain quaternary cyclotomic sequences of period $2p$ was studied in \cite{KIM2} by Kim, Hong and Song, and in \cite{SU}, Su et al. constructed new quaternary sequences of even length with optimal autocorrelation by interleaving certain combinations of binary sequences. In \cite{EDEM}, Edemskiy and Ivanov constructed balanced quaternary sequences of period $2p$ with low autocorrelation, and this construction was further generalized in \cite{SHEN} by Shen et al.
\par
Some recent progress in linear complexity of quaternary sequences over finite fields includes Edemskiy and Ivanov's work on quaternary sequences of period $pq$ in \cite{EDEM4}, and the work of Kim et al. on quaternary sequences constructed from binary Legendre sequences in \cite{KIM3}. A few works concerning linear complexity of binary and quaternary sequences are especially relevant to this correspondence. One of these is by Edemskiy \cite{EDEM1}, in which the linear complexity over $\mathbb{F}_{2}$ of order-four and order-six cyclotomic sequences is investigated. In \cite{EDEM}, Edemskiy and Ivanov were able to compute the linear complexity over $\mathbb{F}_{2^{2}}$ (as well as over $\mathbb{Z}_{4}$) of their above mentioned quaternary sequences. In \cite{WANG}, Wang and Du computed the linear complexity of the binary sequences of order $4l$ constructed in \cite{ACH} by Arasu et al.
\par
This paper focuses on the construction of quaternary sequences with low autocorrelation as well as computing the linear complexity of certain known sequences. We first present a construction that uses cyclotomic classes of order eight to obtain balanced quaternary sequences of odd prime period and low autocorrelation, and investigate the linear complexity over $\mathbb{F}_{2^{2}}$ of some known quaternary sequences of odd period. We also point out a slight modification of a construction given by Chung, Han and Yang in \cite{CHUNG} which applies the inverse Gray-mapping to certain pairs of binary sequences, and then we show that, in certain cases, this slight modification greatly simplifies the construction given by Shen et al. in \cite{SHEN}. Furthermore, we compute the linear complexity over $\mathbb{F}_{2^{2}}$ of these sequences.
\par
The remainder of this paper is organized as follows. In Section \ref{sec2} we introduce some necessary preliminary concepts. In Section \ref{sec6} we construct new balanced quaternary sequences of odd period with low autocorrelation using cyclotomic classes of order eight. We also investigate the linear complexity over $\mathbb{F}_{2^{2}}$ of some known quaternary sequences of odd period. In Section \ref{sec3} we discuss a construction given by Chung et al. in \cite{CHUNG}, and slightly modify it, thereby giving several new families of quaternary sequences with even period and low autocorrelation, as well as investigate their complexity over $\mathbb{F}_{2^{2}}$. Section \ref{sec5} concludes the paper.

\section{Preliminaries}\label{sec2}
Let $\mathbf{s}=(s(t),t=0,1,...,N-1)$ be a sequence of period $N$ over the integer ring $\mathbb{Z}_{m}=\{0,1,...,m-1\}$.  For each $k\in\mathbb{Z}_{m}$, we define $N_{k}(\mathbf{s})=|\{0\leq t<N\mid s(t)=k\}|$. The sequence $\mathbf{s}$ is called {\it balanced} if $max_{k\in\mathbb{Z}_{m}}N_{k}(\mathbf{s})-min_{k\in\mathbb{Z}_{m}}N_{k}(\mathbf{s})\leq 1$. If $\mathbf{s}$ is not balanced, then we say $\mathbf{s}$ is {\it almost balanced} if $max_{k\in\mathbb{Z}_{m}}N_{k}(\mathbf{s})-min_{k\in\mathbb{Z}_{m}}N_{k}(\mathbf{s})\leq2$. For each binary sequence $\mathbf{s}$ of period $N$, the set of all $t\in\mathbb{Z}_{N}$ where $s(t)=1$ is called the {\it support} (or {\it defining set}) of $\mathbf{s}$, and is denoted by $D_{\mathbf{s}}$. Similarly, for each subset $D$ of $\mathbb{Z}_{N}$, the binary sequence of period $N$ whose support is precisely the set $D$ is called the {\it characteristic} sequence of $D$, and is denoted by $\mathbf{s}_{D}$.
\par
Given two sequences $\mathbf{s}_{1}$ and $\mathbf{s}_{2}$ of period $N$ over $\mathbb{Z}_{m}$, the {\it periodic correlation} $R_{\mathbf{s}_{1},\mathbf{s}_{2}}(\tau)$ between $\mathbf{s}_{1}$ and $\mathbf{s}_{2}$ at integer shift $\tau$ for $0\leq\tau<N$ is defined by\[
R_{\mathbf{s}_{1},\mathbf{s}_{2}}(\tau)=\sum_{t=0}^{N-1}\omega^{s_{1}(t)-s_{2}(t+\tau)}\]
where $\omega=e^{\frac{2\pi\sqrt{-1}}{m}}$ is a complex primitive $m$-th root of unity and the addition $t+\tau$ is performed modulo $N$. For $\tau\in \mathbb{Z}_{N}$ fixed, define the {\it left cyclic shift operator} $L^{\tau}(\mathbf{s}_{2})=(s_{2}(t+\tau),t=0,1,...,N-1)$. If $\mathbf{s}_{1}=L^{\tau}(\mathbf{s}_{2})$ for some $\tau\in \mathbb{Z}_{N}$, then $R_{\mathbf{s}_{1},\mathbf{s}_{2}}$ is called the {\it autocorrelation} of the sequence $\mathbf{s}_{1}$ and is denoted by $R_{\mathbf{s}_{1}}$. The values $R_{\mathbf{s}_{1}}(\tau)$, for $0<\tau<N$, are called out-of-phase autocorrelation values. Also define $R_{max}(\mathbf{s}_{1})=max_{0<\tau<N}|R_{\mathbf{s}_{1}}(\tau)|$. We denote the complement of a binary sequence $\mathbf{s}$ by $\overline{\mathbf{s}}$, and that of a set $D$ by $\overline{D}$. Then it is clear that $\mathbf{s}_{\overline{D}}=\overline{\mathbf{s}}_{D}$.
\par
\subsection{Correlation}\label{ssec2.1}
For balanced and almost balanced binary sequences of period $N$, the optimal autocorrelation values can be classified into four categories depending on the value $N$ modulo $4$ \cite[p.~143]{CDR}:
\begin{itemize}
\item $R(\tau)\in \{0,-4\}$ if $N\equiv 0$ $($mod $4)$;
\item $R(\tau)\in \{1,-3\}$ if $N\equiv 1$ $($mod $4)$;
\item $R(\tau)\in \{-2,2\}$ if $N\equiv 2$ $($mod $4)$;
\item $R(\tau)\in \{-1\}$ if $N\equiv 3$ $($mod $4)$.
\end{itemize}
In particular, when $N\equiv 3 \ (\text{mod} \ 4)$, binary sequences with out-of-phase autocorelation $\{-1\}$ are said to have {\it ideal autocorrelation}. Balanced quaternary sequences with optimal autocorrelation property are defined in the following \cite{TD}.
\begin{definition}\label{de1} When $N\equiv 0 \ (\text{mod} \ 2)$, a balanced quaternary sequence $\mathbf{u}$ of period $N$ is said to have {\it optimal autocorrelation magnitude} if $|R_{\mathbf{u}}(\tau)|\leq2$ for all $0\leq\tau<N$.
 \end{definition}
Note that quaternary sequences of odd period $N\equiv 1 \ ({\rm mod} \ 8)$ have been studied in \cite{GREEN}, \cite{HAN}, \cite{YANG} and \cite{YANG1}, where the best out-of-phase autocorrelation magnitude is $3$.
In this paper, we will construct new families of quaternary sequences of both even and odd period having low autocorrelation (some of which have optimal autocorrelation magnitude by Definition \ref{de1}) from pairs of binary sequences with even period and optimal autocorrelation.

\subsection{Gray-mapping}

The mapping $\phi:\mathbb{Z}_{4}\rightarrow\mathbb{Z}_{2}\times \mathbb{Z}_{2}$, commonly referred to as the {\it Gray mapping}, is given by $\phi(0)=(0,0),\phi(1)=(0,1),\phi(2)=(1,1),\phi(3)=(1,0)$. By using the inverse $\phi^{-1}$ of the Gray mapping, every quaternary sequence $\mathbf{u}=(u(t),t=0,1,...,N-1)$ can be obtained from two binary sequences $\mathbf{s}_{1}=(s_{1}(t),t=0,1,...,N-1)$ and $\mathbf{s}_{2}=(s_{2}(t),t=0,1,...,N-1)$ as follows:\[
u(t)=\phi^{-1}(s_{1}(t),s_{2}(t)), \ 0\leq t<N.\]
Krone and Sarwate gave the autocorrelation of $\mathbf{u}$ in terms of the correlations between $\mathbf{s}_{1}$ and $\mathbf{s}_{2}$.
\begin{lemma}\label{le1} {\rm \cite{TD}} The autocorrelation function of $\mathbf{u}$ is given by \[
R_{\mathbf{u}}(\tau)=\frac{1}{2}\left[R_{\mathbf{s}_{1}}(\tau)+R_{\mathbf{s}_{2}}(\tau)\right]+\frac{\omega}{2}
\left[R_{\mathbf{s}_{1},\mathbf{s}_{2}}(\tau)-R_{\mathbf{s}_{2},\mathbf{s}_{1}}(\tau)\right].\]
\end{lemma}
\subsection{Cyclotomic Classes and Cyclotomic Numbers}
Let $q=ef+1$ be a prime power, and $\gamma$ a primitive element of the finite field $\mathbb{F}_{q}$ with $q$ elements. The {\it cyclotomic classes} of order $e$ are given by $D_{i}^{(e,q)}=\gamma^{i}\langle \gamma^{e} \rangle$ for $i=0,1,...,e-1$. Define the {\it cyclotomic numbers of order $e$} by $(i,j)_{e}=|D_{i}^{(e,q)}\cap (D_{j}^{(e,q)}+1)|$. It is easy to see that there are at most $e^{2}$ different cyclotomic numbers of order $e$. When it is clear from the context, we simply denote $(i,j)_{e}$ by $(i,j)$.
The cyclotomic numbers $(h,k)$ of order $e$ have the following properties \cite{D}:

\begin{eqnarray}\label{eq7}
(h,k) & = & (e-h,k-h), \\
(h,k) & = & \begin{cases}
(k,h),                         & \text{if } f \text{ even},\\
(k+\frac{e}{2},h+\frac{e}{2}), & \text{if } f \text{ odd}.
\end{cases}
\end{eqnarray}
\subsection{Linear Complexity of Shift Register Sequences}\label{ssec2.2}
Let $q$ be a prime power. Let $\mathbf{s}=(s(t),t=0,1,...,N-1)$ be a sequence over $\mathbb{F}_{q}$ of period $N$. Define the {\it sequence polynomial} of the sequence $\mathbf{s}$ as \[
\mathbf{s}(x)=s(0)+s(1)x+\cdots+s(N-1)x^{N-1}\in\mathbb{F}_{q}\left[x\right].\]It is known \cite[p.~273]{SOSL} that the {\it minimal polynomial} $P_{\mathbf{s}}(x)$ of the sequence $\mathbf{s}$ is given by \begin{equation}\label{eqn4}
P_{\mathbf{s}}(x)=(x^{N}-1)/\gcd(x^{N}-1,\mathbf{s}(x)),\end{equation}
and that the linear complexity $L(\mathbf{s})$ of the sequence $\mathbf{s}$ is the degree of the minimal polynomial $P_{\mathbf{s}}(x)$, i.e., \begin{equation}\label{eqn5}
L(\mathbf{s})=N-\text{deg}(\gcd(x^{N}-1,\mathbf{s}(x))).\end{equation}

\section{Balanced Quaternary Sequences of Odd Period with Low Autocorrelation}\label{sec6}

\begin{table}\label{ta2}
\begin{center}
\tabcolsep=0.1cm
\small
\begin{tabular}{|c|c|c|c|c|c|c|c|c|}
\hline
  $(j,i)$& 0&1&2&3&4&5&6&7 \\
  \hline
0&$(0,0)$&$(0,1)$&$(0,2)$&$(0,3)$&$(0,4)$&$(0,5)$&$(0,6)$&$(0,7)$\\\hline
1&$(0,1)$&$(0,7)$&$(1,2)$&$(1,3)$&$(1,4)$&$(1,5)$&$(1,6)$&$(1,2)$\\\hline
2&$(0,2)$&$(1,2)$&$(0,6)$&$(1,6)$&$(2,4)$&$(2,5)$&$(2,4)$&$(1,3)$\\\hline
3&$(0,3)$&$(1,3)$&$(1,6)$&$(0,5)$&$(1,5)$&$(2,5)$&$(2,5)$&$(1,4)$\\\hline
4&$(0,4)$&$(1,4)$&$(2,4)$&$(1,5)$&$(0,4)$&$(1,4)$&$(2,4)$&$(1,5)$\\\hline
5&$(0,5)$&$(1,5)$&$(2,5)$&$(2,5)$&$(1,4)$&$(0,3)$&$(1,3)$&$(1,6)$\\\hline
6&$(0,6)$&$(1,6)$&$(2,4)$&$(2,5)$&$(2,4)$&$(1,3)$&$(0,2)$&$(1,2)$\\\hline
7&$(0,7)$&$(1,2)$&$(1,3)$&$(1,4)$&$(1,5)$&$(1,6)$&$(1,2)$&$(0,1)$\\\hline
\multicolumn{8}{l}{{\scriptsize }}\\
\end{tabular}
\caption{{\scriptsize The relations of cyclotomic numbers of order 8 modulo a prime $p$ for the case where $p\equiv 1 \ ({\rm mod} \ 16)$.}}
\end{center}
\end{table}

\begin{table}\label{ta3}
\begin{center}
\tabcolsep=0.1cm
\small
\begin{tabular}{cl|cl}

\hline
  \hline
$64(0,0)$&$p-23+6x$&$64(1,2)$&$p+1-6x+4a$\\\hline
$64(0,1)$&$p-7+2x+4a$&$64(1,3)$&$p+1+2x-4a-16b$\\\hline
$64(0,2)$&$p-7-2x-8a-16y$&$64(1,4)$&$p+1+2x-4a+16y$\\\hline
$64(0,3)$&$p-7+2x+4a$&$64(1,5)$&$p+1+2x-4a-16y$\\\hline
$64(0,4)$&$p-7-10x$&$64(1,6)$&$p+1+2x-4a+16b$\\\hline
$64(0,5)$&$p-7+2x+4a$&$64(2,4)$&$p+1+6x+8a$\\\hline
$64(0,6)$&$p-7-2x-8a+16y$&$64(2,5)$&$p+1-6x+4a$\\\hline
$64(0,7)$&$p-7+2x+4a$&\multicolumn{2}{l}{{\scriptsize }}\\\hline
\multicolumn{4}{l}{{\scriptsize }}\\

\end{tabular}
\caption{{\scriptsize The cyclotomic numbers of order 8 modulo a prime $p$ for the case where $p\equiv 1 \ ({\rm mod} \ 16)$.}}
\end{center}
\end{table}
\subsection{New Balanced Quaternary Sequences of Odd Period with Low Autocorrelation}

In this section we construct new quaternary sequences of odd period and low autocorrelation using cyclotomic classes of order eight. For convenience, we will denote the cyclotomic classes $D_{i}^{(8,p)}$ of order eight modulo a prime $p$, simply by $D_{i}$.

\begin{theorem}\label{th68} Let $p=x^{2}+16=a^{2}+2b^{2}\equiv 1 \ ({\rm mod} \ 16) \ (x\equiv a\equiv 1 \ ({\rm mod} \ 4))$ be a prime such that $x-a=4$. Define $C_{0}=D_{2}\cup D_{6},C_{1}=D_{1}\cup D_{3},C_{2}=D_{0}\cup D_{4}$ and $C_{3}=D_{5}\cup D_{7}$, and let $\mathbf{u}$ be the quaternary sequence of period $p$ defined by
\[u(t)=\begin{cases} 0, & \text{ if } t \in C_{0}\cup\{0\},\\
                     1, & \text{ if } t \in C_{1},\\
                     2, & \text{ if } t \in C_{2},\\
                     3, & \text{ if } t \in C_{3}.\\\end{cases}\]Then \[N_{j}(\mathbf{u})=\begin{cases} \frac{p+3}{4},&\text{ if } j=0,\\
                                                                          \frac{p-1}{4},&\text{ otherwise,}\end{cases}\text{ and } \
    R_{\mathbf{u}}(\tau)=\begin{cases} p,& \text{once,}\\
                                          -1,&\frac{p-1}{8}\text{ times,}\\
                                          -3,&\frac{p-1}{2}\text{ times,}\\
                                           3,&\frac{3(p-1)}{8}\text{ times.}\end{cases}\]

\end{theorem}
\begin{proof} Assume that $\tau^{-1}\in D_{h}$, for $h=0,...,7$, and let $\zeta_{z}^{h}=1$ if $z\in D_{h}$ and $\zeta_{z}^{h}=0$ otherwise. The real part of the autocorrelation is given by \begin{eqnarray*} Re(R_{\mathbf{u}}(\tau))
& = & \mid C_{h}\cap (C_{h}+1) \mid  + \mid C_{1+h}\cap (C_{1+h}+1) \mid + \mid C_{2+h}\cap (C_{2+h}+1) \mid\\
&   & + \mid C_{3+h}\cap (C_{3+h}+1) \mid-\mid C_{h}\cap (C_{2+h}+1) \mid  - \mid C_{2+h}\cap (C_{h}+1) \mid \\
&   & - \mid C_{1+h}\cap (C_{3+h}+1) \mid- \mid C_{3+h}\cap (C_{1+h}+1) \mid\\
&   & +\zeta_{1}^{6+h}+\zeta_{-1}^{6+h}+\zeta_{1}^{2+h}+\zeta_{-1}^{2+h}-\zeta_{-1}^{4+h}-\zeta_{-1}^{h}-\zeta_{1}^{4+h}-\zeta_{1}^{h}\\
& = & \sum_{j=0}^{7}(j+h,j+h)+2\left[(6+h,2+h)+(1+h,3+h)+(4+h,h)+(5+h,7+h)\right]\\
&   & -2[(6+h,4+h)+(6+h,h)+(2+h,4+h)+(2+h,h)\\
&   & +(1+h,5+h)+(1+h,7+h)+(3+h,5+h)+(3+h,7+h)]\\
&   & +\zeta_{1}^{6+h}+\zeta_{-1}^{6+h}+\zeta_{1}^{2+h}+\zeta_{-1}^{2+h}-\zeta_{-1}^{4+h}-\zeta_{-1}^{h}-\zeta_{1}^{4+h}-\zeta_{1}^{h}.\\\end{eqnarray*} Using Tables II and III (which can be found in \cite{CDR}) we are able to compute the values $Re(R_{\mathbf{u}}(\tau))$ as $h$ runs over $\{0,...,7\}$. We list these values in tabular form below.

\begin{center}
\begin{tabular}{c||cccccccc}
$h$&0&1&2&3&4&5&6&7\\\hline
$Re(R_{\mathbf{u}}(\tau))$&$-3$&$\frac{x-a+2}{2}$&$a-x+1$&$\frac{x-a+2}{2}$&$-3$&$\frac{x-a+2}{2}$&$a-x+1$&$\frac{x-a-4}{2}$\\
\end{tabular}
\end{center}
The imaginary part of the autocorrelation is given by
\begin{eqnarray*} Im(R_{\mathbf{u}}(\tau))
& = & \mid C_{1+h}\cap (C_{h}+1) \mid  + \mid C_{2+h}\cap (C_{1+h}+1) \mid + \mid C_{3+h}\cap (C_{2+h}+1) \mid\\
&   & + \mid C_{h}\cap (C_{3+h}+1) \mid-\mid C_{h}\cap (C_{1+h}+1) \mid  - \mid C_{1+h}\cap (C_{2+h}+1) \mid \\
&   & - \mid C_{2+h}\cap (C_{3+h}+1) \mid- \mid C_{3+h}\cap (C_{h}+1) \mid\\
&   & +\zeta_{1}^{1+h}+\zeta_{1}^{3+h}+\zeta_{-1}^{5+h}+\zeta_{-1}^{7+h}-\zeta_{-1}^{1+h}-\zeta_{-1}^{3+h}-\zeta_{1}^{5+h}-\zeta_{1}^{7+h}.\\
\end{eqnarray*} Using the symmetry of Table II, it is easy to see that, after expanding the first eight terms, they cancel with each other. That the last eight terms cancel with each other is due to the fact that 1 and -1 are always members of $D_{0}$. Thus the autocorrelation values $R_{\mathbf{u}}(\tau)$ and the values $N_{j}(\mathbf{u})$ are as in the statement of the theorem.
\end{proof}
Note that the first several primes satisfying the conditions of Theorem \ref{th68} are 17, 97, 641, 2417, 6577 and 14,657.
\begin{example} Let $p=17$. Then the quaternary sequence obtained by Theorem \ref{th68} is given by $\mathbf{u}=02012331001332102...$ and has out-of-phase ($\tau\not\equiv 0 \ ({\rm mod} \ p)$) autocorrelation values $R_{\mathbf{u}}(\tau)\in\{-1,\pm3\}$.
\end{example}

\subsection{Linear Complexity over $\mathbb{F}_{2^{2}}$ of Some Known Quaternary Sequences of Odd Period}
In this section, we investigate the linear complexity of some quaternary sequences that were constructed by Tang and Lindner in \cite{TANG1}. Let $p=4f+1=a^{2}+4b^{2}$ be a prime with $a,b \in \mathbb{Z}$ and $a \equiv 1$ (mod 4) (here, $b$ is two-valued depending on the choice of the primitive root $\alpha$ defining the cyclotomic classes).  Recall from Section \ref{ssec2.2} that the minimal polynomial $P_{\mathbf{s}}(x)$ and the linear complexity $L(\mathbf{s})$ of a sequence $\mathbf{s}$ over $\mathbb{F}_{2^{m}}$ are given by (\ref{eqn4}) and (\ref{eqn5}) respectively. Let $\mathbb{F}_{2^{2}}=\mathbb{F}_{2}\mu+\mathbb{F}_{2}$ be the finite field with four elements, where $\mu$ satisfies the relation $\mu^{2}+\mu+1=0$. For convenience, we denote the cyclotomic classes $D_{i}^{(4,p)}$ of order four simply by $D_{i}$.
\par
The following lemma is a known construction of quaternary sequences. We give a proof (see Appendix), for the convenience of the reader, in terms of its binary sequence-pair representation, which is the more favorable representation for computing the complexity.
\begin{lemma}\label{th67} {\rm \cite{TANG1}} Define $C_{0}=D_{i}\cup D_{j}$ and $C_{1}=D_{j}\cup D_{l}\cup\{0\}$ where $i,j$ and $l$ are distinct, and define a quaternary sequence $\mathbf{u}$ of period $p$ by $u(t)=\phi^{-1}(s_{C_{0}}(t),s_{C_{1}}(t))$. Then \[N_{j}(\mathbf{u})=\begin{cases} \frac{p+3}{4},&\text{ if } j=0,\\
                                                                          \frac{p-1}{4},&\text{ otherwise,}\end{cases}\]and \[
    R_{\mathbf{u}}(\tau)=\begin{cases} p,& \text{once,}\\
                                          -1,&\frac{p-1}{2}\text{ times,}\\
                                          1,&\frac{p-1}{4}\text{ times,}\\
                                          -3,&\frac{p-1}{4}\text{ times,}\end{cases}\text{ resp. } \
    R_{\mathbf{u}}(\tau)\begin{cases} p,& \text{once,}\\
                                          -1+2\omega,&\frac{p-1}{4}\text{ times,}\\
                                          1+2\omega,&\frac{p-1}{4}\text{ times,}\\
                                          -1,&\frac{p-1}{2}\text{ times,}\end{cases}\]whenever $f$ is even and $(i,j,l)\in\{(1,2,3),(1,3,0)\}$ resp. $f$ is odd and $(i,j,l)\in\{(1,2,3)\}$.

\end{lemma}
\begin{proof} See appendix.
\end{proof}

We will show the following.
\begin{theorem}\label{th60} Let the sequence $\mathbf{u}$ be defined as in Lemma \ref{th67}. If $p\equiv 1 \ ($mod $8)$ then the linear complexity over $\mathbb{F}_{2^{2}}$ of $\mathbf{u}$ is $L(\mathbf{u})=\frac{p-1}{2}$. If $p\equiv 5 \ ($mod $8)$, then $L(\mathbf{u})=p-1$.
\end{theorem}
Define $S_{d}(x)=\sum_{t\in D_{0}^{(d,p)}}x^{t}$, and let $\beta$ be a $p$th root of unity over $\mathbb{F}_{2^{m}}$.
 The following lemma was shown by Edemskiy in his work on the linear complexity of binary quartic and sextic power residue sequences.

\begin{lemma}\label{le15} {\rm \cite{EDEM1}} \begin{enumerate}[(1)] For $d=2,4$, the following hold:
\item $\sum_{t\in D_{r}}\beta^{t}=S_{d}(\beta^{\alpha^{r}})$ for $r=0,1,...,d-1$;
\item $S_{d}(\beta^{\alpha^{r+dg}})=S_{d}(\beta^{\alpha^{r}})$ for any integer $g$;
\item if $2\in D_{l}^{(d,p)}$ then $S_{d}(\beta^{\alpha^{l}})=S_{d}^{2}(\beta)$;
\item if $2\in D_{0}^{(d,p)}$ then $S_{d}(\beta^{\alpha^{g}})\in \{0,1\}$ for any integer $g$;
\item $S_{d}(\beta^{\alpha^{r}})\neq 1$ for at least one $r=0,1,...,d-1$.
\end{enumerate}
\end{lemma}
Note that \[
0=\beta^{p}-1=(\beta-1)(1+\beta+\cdots+\beta^{p-1}),\] so that, by Lemma \ref{le15}, \begin{equation}\label{eqn3}
S_{4}(\beta)+S_{4}(\beta^{\alpha})+S_{4}(\beta^{\alpha^{2}})+S_{4}(\beta^{\alpha^{3}})=1. \end{equation} The following lemma was also shown by Edemskiy in the above mentioned work.
\begin{lemma}\label{le16} {\rm \cite{EDEM1}} Let $f$ ($=\frac{p-1}{4}$) be even, $\gamma$ be a root of the equation $x^{2}+x+1=0$, and define $\mathbf{S}=(S_{4}(\beta),S_{4}(\beta^{\alpha}),S_{4}(\beta^{\alpha^{2}}),S_{4}(\beta^{\alpha^{3}}))$. Then
\begin{enumerate}[(1)]
\item $\mathbf{S}=(1,0,0,0)$ or $(0,0,1,0)$ if $a\equiv 1 \ ($mod $8)$ and $b\equiv 0 \ ($mod $4)$;
\item $\mathbf{S}=(1,1,0,1)$ or $(0,1,1,1)$ if $a\equiv 5 \ ($mod $8)$ and $b\equiv 0 \ ($mod $4)$;
\item $\mathbf{S}=(\gamma,1,\gamma+1,1)$ if $a\equiv 1 \ ($mod $8)$ and $b\equiv 2 \ ($mod $4)$;
\item $\mathbf{S}=(\gamma,0,\gamma+1,0)$ if $a\equiv 5 \ ($mod $8)$ and $b\equiv 2 \ ($mod $4)$;
\end{enumerate}
\end{lemma}
\begin{table}\label{ta1}
\begin{center}
\tabcolsep=0.1cm
\footnotesize
\begin{tabular}{|c|c|c|c|c|}
\hline
  Period $N$ &$R_{max}$&Complexity $L$ over $\mathbb{F}_{2^{2}}$&Construction&Complexity \\
  \hline
  $N=pq\equiv 5 \ ({\rm mod} \ 8)$ and $p-q=4$  & $3$&unknown&\cite{GREEN}&NA\\\hline
  $N=pq\equiv 5 \ ({\rm mod} \ 8)$ and $p-q=4$  & $3$&unknown&\cite{HAN}&NA\\\hline
  $N=pq\equiv 3 \ ({\rm mod} \ 4)$ and $p-q=2$  & $\sqrt{5}$&\shortstack{Many cases, e.g. if $p\equiv1 \ ({\rm mod} \ 4)$ and \\$pq\equiv-1 \ ({\rm mod} \ 8)$ then $L=pq$}&\cite{YANG}&\cite{EDEMS}\\\hline
  $N=p\equiv 5 \ ({\rm mod} \ 8)$ resp. $1 \ ({\rm mod} \ 8)$  & $\sqrt{5}$ resp. $3$&unknown&\cite{SID}&NA\\\hline
  $N=p\equiv 5 \ ({\rm mod} \ 8)$ resp. $1 \ ({\rm mod} \ 8)$  & $\sqrt{5}$ resp. $3$&$L=p-1$ resp. $\frac{p-1}{2}$&\cite{TANG1}&Theorem \ref{th60}\\\hline
  $N=p\equiv 5 \ ({\rm mod} \ 8)$ resp. $1 \ ({\rm mod} \ 8)$  & $\sqrt{5}$ resp. $3$&unknown&\cite{GREEN1}&NA\\\hline
  $N=p\equiv 1 \ ({\rm mod} \ 8)$  & $\sqrt{5}$&unknown&\cite{YANG1}&NA\\\hline
  $N=p\equiv 1 \ ({\rm mod} \ 16)$  & $3$&unknown&Theorem \ref{th68}&NA\\\hline
  \multicolumn{3}{l}{{\scriptsize Note: $p,q$ are odd primes.}}\\
\end{tabular}
\caption{{\scriptsize Comparison of Balanced Quaternary Sequences of Odd Period with Low Autocorrelation}}
\end{center}
\end{table}
We now give the proof of Theorem \ref{th60}.
\begin{proof} We again only show the case $(i,j,l)=(1,2,3)$ as the other cases are almost identical. Then it is easy to deduce that \[
u(t)=\begin{cases} 0,& \text{ if } t \ (\text{mod} p)\in D_{0}\cup\{0\},\\
                   1,& \text{ if } t \ (\text{mod} p)\in D_{3},\\
                   \mu+1,& \text{ if } t \ (\text{mod} p)\in D_{2},\\
                   \mu,& \text{ if } t \ (\text{mod} p)\in D_{1},\end{cases}\text{ whence } \mathbf{u}(x) = \sum_{t\in D_{2}\cup D_{3}}x^{t}+\mu\sum_{t\in D_{1}\cup D_{2}}x^{t}.\]

It is sufficient to find the number of roots of $\mathbf{u}(x)$ in the set $\{\beta^{v}\mid v=0,1,...,p-1\}$. We have $\mathbf{u}(1)=0$ so that $1$ is always a zero. We can write \[
 \mathbf{u}(\beta^{v})=S_{4}(\beta^{v\alpha^{2}})+S_{4}(\beta^{v\alpha^{3}})+\mu(S_{4}(\beta^{v\alpha})+S_{4}(\beta^{v\alpha^{2}})),\] which gives us \[
\mu \mathbf{u}(\beta^{v})=\mu(S_{4}(\beta^{v\alpha^{1}})+S_{4}(\beta^{v\alpha^{3}}))+\mu(S_{4}(\beta^{v\alpha})+\mu(S_{4}(\beta^{v\alpha^{2}}).\] Let $f$ ($=\frac{p-1}{4}$) first be even. By Lemma \ref{le15} and (\ref{eqn3}) we can, without loss of generality, assume that $S_{2}(\beta)=1$, and $S_{2}(\beta^{\alpha})=0$. Thus, by Lemma \ref{le16}, $\mu \mathbf{u}(\beta^{v})=S_{4}(\beta^{v\alpha})+S_{4}(\beta^{v\alpha^{2}})$ is equal to\[
\begin{cases} S_{4}(\beta^{\alpha})+S_{4}(\beta^{\alpha^{2}}),&\text{ if }v\in D_{0},\\
               S_{4}(\beta)+S_{4}(\beta^{\alpha^{3}}),&\text{ if }v\in D_{2},\\
               \mu+S_{4}(\beta^{\alpha^{2}})+S_{4}(\beta^{\alpha^{3}}),&\text{ if }v\in D_{1},\\
               \mu+S_{4}(\beta)+S_{4}(\beta^{\alpha}),&\text{ if }v\in D_{3},\end{cases}=\begin{cases}
0,& \text{only if $v\in D_{0}$, or only if $v\in D_{2}$}\\ &\text{for $b\equiv 0 \ (\text{mod } 4)$},\\
0,& \text{only if $v\in D_{1}$, or only if $v\in D_{3}$}\\ &\text{for $b\equiv 2 \ (\text{mod } 4)$}.\end{cases}\] Thus we have $P_{\mathbf{u}}(x)=(x^{p}-1)/H(x)$ where $H(x)=\sum_{D_{i}\cup\{0\}}x^{t}$ where $i=0$ or $2$ when $b\equiv 0 \ (\text{mod } 4)$, and $i=1$ or $3$ when $b\equiv 2 \ (\text{mod } 4)$. The result for when $f$ is even follows.
\par
Now let $f$ be odd. By (4) of Lemma \ref{le15}, we can, without loss of generality, assume that $S_{4}(\beta^{v\alpha})+S_{d}(\beta^{v\alpha^{3}})=\mu+1$. Then we have $\mu \mathbf{u}(\beta^{v})=1+S_{4}(\beta^{v\alpha})+S_{4}(\beta^{v\alpha^{2}})$. We can also assume, without loss of generality, that $2\in D_{1}$ (if not, then we can simply replace $\alpha$ with $\alpha^{-1}$). Define $S(x)=S_{4}(x^{\alpha})+S_{4}(x^{\alpha^{2}})$. By (3) of Lemma \ref{le15} we have that $S_{4}^{2}(\beta)=S_{4}(\beta^{\alpha})$ whence $S^{2}(\beta)=S(\beta^{\alpha})$. Now suppose that $\mu \mathbf{u}(\beta^{v})=0$ for some $v=1,...,p-1$. Then we have $S(\beta^{v})=1$. Then $1=S^{2}(\beta^{v})=S(\beta^{\alpha})$. This gives us $S(\beta^{\alpha^{l}})=1$ for all $l=1,...,p-1$. Then $1=S_{4}(\beta^{\alpha})+S_{4}(\beta^{\alpha^{2}})=S_{4}(\beta^{\alpha^{2}})+S_{4}(\beta^{\alpha^{3}})$. But this means that $S_{4}(\beta^{\alpha^{1}})+S_{4}(\beta^{\alpha^{3}})=0$, which is contrary to our assumption. Thus we have $\mathbf{u}(\beta^{v})\neq0$ for  all $v\neq0$, and the result for when $f$ is odd follows.
\end{proof}
\begin{remark} The linear complexity over $\mathbb{F}_{2^{2}}$ of the sequences constructed in Theorem \ref{th68} has yet to be examined theoretically. We have, however, checked this complexity numerically with MAGMA using the Berlekamp-Massey algorithm on the first several primes meeting the conditions. These calculations are given in the table below.
\begin{center}
\begin{tabular}{c||cccccc}
$p$&$17$&$96$&$641$&$2417$&$6577$&$14657$\\\hline
$L(\mathbf{u})$&$8$&$48$&$320$&$1208$&$3288$&$7328$\\
\end{tabular}
\end{center} This suggests that the linear complexity over $\mathbb{F}_{2^{2}}$ of the sequences constructed in Theorem \ref{th68} is the same as that of the sequences constructed by Tang and Linder in \cite{TANG1} having the same period and autocorrelation magnitude. From Table \ref{ta1} it is also clear that the quaternary sequences constructed in \cite{GREEN1} and in \cite{SID} of the same period also have the same autocorrelation magnitude as those obtained by Theorem \ref{th68}. However, by a simple examination of the autocorrelation distributions, it is easy to see than the quaternary sequences obtained by Theorem \ref{th68} cannot be equivalent to any of those constructed in \cite{GREEN1}, \cite{SID} or \cite{TANG1}.
\end{remark}
\section{New Balanced and Almost Balanced Quaternary Sequences of Even Period with Low Autocorrelation}\label{sec3}

In this section we discuss the quaternary sequences constructed by Chung et al. in \cite{CHUNG}. We point out a slight modification which leads to new families of quaternary sequences with even period and low autocorrelation, and we show that the quaternary sequences constructed in \cite{SHEN} are the same as those obtained by applying this modification to certain known binary sequences. We also compute the linear complexity over $\mathbb{F}_{2^{2}}$ of the sequences discussed in this section.
\par
Let $\mathbf{s}$ be a binary sequence with even period $N$ and low autocorrelation. The main idea behind Chung et al.'s construction is to apply Lemma \ref{le1} to the binary sequence pair $(\mathbf{s},L^{\frac{N}{2}}(\mathbf{s}))$. The same idea can also be applied to the sequence pair $(\mathbf{s},L^{\frac{N}{2}}(\mathbf{\overline{s}}))$.

\begin{theorem}\label{le2} Let $\mathbf{s}_{0}$ be a binary sequence of even period $N$, and let $\mathbf{s}_{1}$ denote either $L^{\frac{N}{2}}(\mathbf{s}_{0})$ or $L^{\frac{N}{2}}(\mathbf{\overline{s}}_{0})$. Define a quaternary sequence $\mathbf{u}$ of the same period by $u(t)=\phi^{-1}(s_{0}(t),s_{1}(t))$. Then $R_{\mathbf{u}}(\tau)=R_{\mathbf{s}_{0}}(\tau)$ for all $\tau,0\leq\tau<N$. Moreover, we have\[
N_{i}(\mathbf{u})=\begin{cases} \mid \overline{D}_{\mathbf{s}_{0}}\cap (\overline{D}_{\mathbf{s}_{1}}-\frac{N}{2})\mid, \text{ if } i=0,\\
                       \mid \overline{D}_{\mathbf{s}_{0}}\cap (D_{\mathbf{s}_{1}}-\frac{N}{2})\mid, \text{ if } i=1,\\
                       \mid D_{\mathbf{s}_{0}}\cap (D_{\mathbf{s}_{1}}-\frac{N}{2})\mid, \text{ if } i=2,\\
                       \mid D_{\mathbf{s}_{0}}\cap (\overline{D}_{\mathbf{s}_{1}}-\frac{N}{2})\mid, \text{ if } i=3.\end{cases}\]
\end{theorem}
\begin{proof} The case where $\mathbf{s}_{1}=L^{\frac{N}{2}}(\mathbf{s}_{0})$ was shown in \cite{CHUNG}. We show the case where $\mathbf{s}_{1}=L^{\frac{N}{2}}(\mathbf{\overline{s}}_{0})$. It is clear that $R_{\mathbf{s}_{0}}(\tau)=R_{L^{\frac{N}{2}}(\overline{\mathbf{s}}_{0})}(\tau)$ for all $\tau,0\leq\tau<N$. We need only show that $R_{\mathbf{s}_{0},L^{\frac{N}{2}}(\overline{\mathbf{s}}_{0})}(\tau)=R_{L^{\frac{N}{2}}(\overline{\mathbf{s}}_{0}),\mathbf{s}_{0}}(\tau)$ for all $\tau,0\leq\tau<N$. Notice that, for all $t$ and for all $\tau,0\leq\tau<N$, we have $s_{0}(t)=\overline{s}_{0}(t+\frac{N}{2}+\tau)$ whenever $\overline{s}_{0}(t)=s_{0}(t+\frac{N}{2}+\tau)$, and the latter holds whenever $\overline{s}_{0}(t+\frac{N}{2})=s_{0}(t+\tau)$.
Then we have \begin{eqnarray*} R_{\mathbf{s}_{0},L^{\frac{N}{2}}(\overline{\mathbf{s}}_{0})}(\tau) & = & \sum_{t=0}^{N-1}(-1)^{s_{0}(t)-\overline{s}_{0}(t+\frac{N}{2}+\tau)} \\
                                                                                 & = & \sum_{t=0}^{N-1}(-1)^{\overline{s}_{0}(t+\frac{N}{2})-s_{0}(t+\tau)} \\
                                                                                 & = & R_{L^{\frac{N}{2}}(\overline{\mathbf{s}}_{0}),\mathbf{s}_{0}}(\tau). \end{eqnarray*}
Thus, by Lemma \ref{le1}, we have that $R_{\mathbf{u}}(\tau)=R_{\mathbf{s}_{0}}(\tau)$ for all $\tau,0\leq\tau<N$. Since $u(t)=i$ if and only if $s(t)=\phi(i)$, computation of the $N_{i}(\mathbf{u})$'s comes from counting the number of $t$s such that $s(t)=\phi(i)$ for $i=0,1,2,3$.
\end{proof}
We also have the following theorem concerning the balancedness of the sequences obtained in Theorem \ref{le2} is immediate after replacing $b_{0}(t+N/2)$ by $\overline{b}_{0}(t+N/2)$ in Lemma 7 of \cite{CHUNG}, and so the proof is omitted.
\begin{theorem}\label{off1} Let $\mathbf{s}_{0}$ and $\mathbf{u}$ be defined as in Theorem \ref{le2}. Then \begin{enumerate}[(i)]
\item if $\mathbf{s}_{0}$ is balanced, then $\mathbf{u}$ is balanced if $N\equiv 0,2,6 \ ({\rm mod} \ 8)$, and almost balanced if $N\equiv 4 \ ({\rm mod} \ 8)$;
\item if $\mathbf{s}_{0}$ is almost balanced, then $\mathbf{u}$ is almost balanced if $N\equiv 2,4,6 \ ({\rm mod} \ 8)$.
\end{enumerate}
\end{theorem}

One of the reasons we wish to point out the modifications given in Theorems \ref{le2} and \ref{off1}, is to show that the following known construction can be viewed as a special case.
\begin{lemma}\label{le11} {\rm \cite{SHEN}} Let $p\equiv 5$ $($mod $8)$ be a prime. Define $C_{k}=\psi(\{0\}\times D_{i_{k}}^{(4,p)})\cup\psi(\{1\}\times D_{j_{k}}^{(4,p)})$ for $k=0,1,2,3$, where $i_{k}\neq i_{m}$ and $j_{k}\neq j_{m}$ if $k\neq m$. Now let $H_{k}=C_{k}$ for $k=1,3$, $H_{0}=C_{0}\cup \{0\}$, and $H_{2}=C_{2}\cup\{p\}$. Define a quaternary sequence $\mathbf{u}$ by $u(t)=k$ whenever $t \ ($mod $2p)\in H_{k}$. Then $\mathbf{u}$ is balanced and $R_{\mathbf{u}}(\tau)\in\{-2,2\}$ for $\tau\neq0$.
\end{lemma}
The sequences constructed in Lemma \ref{le11} can be obtained by applying Theorems \ref{le2} and \ref{off1} to the Ding-Helleseth-Martinsen binary sequences constructed in \cite{DHM}.
\begin{theorem}\label{th5} The quaternary sequences constructed by Lemma \ref{le11} are the same as those constructed by applying Theorems \ref{le2} and \ref{off1} to the Ding-Helleseth-Martinsen binary sequences constructed in \cite{DHM}.
\end{theorem}
\begin{proof} We know that $p=a^{2}+4b^{2}$ for some $a$ and $b$ with $x\equiv \pm1$ $($mod $4)$. Let $D$ be defined as in \cite[Theorem 5.11]{CUN}, and let $H_{k},k=0,1,2,3$, be defined as in Lemma \ref{le11}. To see the equivalence, one needs only check the following equalities as $(i,j,l)$ runs over $\{(0,1,3),(0,2,3),(1,2,0),(1,3,0)\}$ when $a=1$, or as $(i,j,l)$ runs over $\{(0,1,2),(0,3,2),(1,0,3),(1,2,3)\}$ when $b=1$:\begin{equation}\label{eq69}
H_{0}=\psi(\overline{D})\cap(\psi(D)-p),H_{1}=\psi(\overline{D})\cap(\psi(\overline{D})-p),H_{2}=\psi(D)\cap(\psi(\overline{D})-p)\text{ and } H_{3}=\psi(D)\cap(\psi(D)-p).\end{equation}This concludes the proof.
\end{proof}
\begin{remark} Note that the equivalence shown in Theorem \ref{th5} can only be realized with the small modification to the construction given in \cite{CHUNG} pointed out in Theorem \ref{le2}. It is not difficult to check that the equalities in (\ref{eq69}) do not hold when one uses the construction as given in \cite{CHUNG}.
\end{remark}
\begin{example} Let $p=5$, and $(i,j,l)=(0,1,2)$. Then \[D=\left[\{0\}\times(D_{0}^{(4,p)}\cup D_{1}^{(4,p)})\right]\cup\left[\{1\}\times(D_{1}^{(4,p)}\cup D_{2}^{(4,p)})\right]\cup\{(0,0)\},\] and we have $\mathbf{s}_{D}=1010001101...$ and $\mathbf{u}=\phi^{-1}(\mathbf{s}_{D},L^{p}(\overline{\mathbf{s}}_{D}))=2031002312...$.
\end{example}
\subsection{Linear Complexity Over $\mathbb{F}_{2^{2}}$ of Quaternary Sequences Discussed in this Section}\label{ssec2}
 In this section we will view a polynomial in $\mathbb{F}_{2}\left[x\right]$ as a polynomial in $\mathbb{F}_{2^{2}}\left[x\right]$ in the sense that $\mathbb{F}_{2^{2}}\left[x\right]=(\mathbb{F}_{2}\mu+\mathbb{F}_{2})\left[x\right]$, where $\mu$ satisfies the relation $\mu^{2}+\mu+1=0$. Recall from Section \ref{ssec2.2} that the minimal polynomial $P_{\mathbf{s}}(x)$ and the linear complexity $L(\mathbf{s})$ of a sequence $\mathbf{s}$ over $\mathbb{F}_{2^{m}}$ are given by (\ref{eqn4}) and (\ref{eqn5}) respectively. For two binary sequences $\mathbf{s}_{1}$ and $\mathbf{s}_{2}$ of period $l$, (by using the inverse Gray-map) we can obtain a sequence $\mathbf{v}$, defined by $v(t)=s_{1}(t)\mu +s_{2}(t)$, over $\mathbb{F}_{2^{2}} \ (=\mathbb{F}_{2}\mu+\mathbb{F}_{2}$). Let \[
\mathbf{s}_{1}(x)=s_{1}(0)+s_{1}(1)x+\cdots+s_{1}(l-1)x^{l-1}\text{ and }\mathbf{s}_{2}(x)=s_{2}(0)+s_{2}(1)x+\cdots+s_{2}(l-1)x^{l-1}\]be the sequence polynomials in $\mathbb{F}_{2}\left[x\right]$ of $\mathbf{s}_{1}$ and $\mathbf{s}_{2}$, respectively. Then the sequence polynomial $\mathbf{v}(x)\in\mathbb{F}_{2^{2}}\left[x\right]$ of $\mathbf{v}$ is given by \begin{equation}\label{eqn1}
\mathbf{s}_{1}(x)\mu+\mathbf{s}_{2}(x).\end{equation}
The following lemma allows us to treat the minimal polynomial of certain pairs of binary sequences.

\begin{lemma}\label{le13} {\rm \cite{LN} \cite{UI}} Suppose that $P_{\mathbf{s}}(x)$ resp. $P_{\mathbf{v}}(x)$ are the minimal polynomials of the binary sequences of period $N$, $\mathbf{s}$ resp. $\mathbf{v}$. If $\mathbf{s}$ can be obtained from $\mathbf{v}$ by a cyclic shift by $\delta$, then $P_{\mathbf{s}}(x)=P_{\mathbf{v}}(x)$. If the sequence $\mathbf{v}$ is the complement of $\mathbf{s}$ then \[
P_{\mathbf{v}}(x)=\begin{cases} P_{\mathbf{s}}(x)(x-1),& \text{ if } (x-1) \nmid P_{\mathbf{s}}(x),\\
                       P_{\mathbf{s}}(x)/(x-1),& \text{ if } (x-1) \mid P_{\mathbf{s}}(x)\text{ but } (x-1)^{2} \nmid P_{\mathbf{s}}(x),\\
                       P_{\mathbf{s}}(x),& \text{ if } (x-1)^{2} \mid P_{\mathbf{s}}(x).\end{cases}\]
\end{lemma}
\begin{table}\label{ta1}
\begin{center}
\tabcolsep=0.1cm
\footnotesize
\begin{tabular}{|c|c|c|c|c|}
\hline
  Period $N$ &$R_{max}$&Complexity $L$ over $\mathbb{F}_{2^{2}}$&Construction ref&Complexity ref \\
  \hline
  $N=p^{n}-1$ & $\sqrt{8},4$&unknown&\cite{JANG}&NA\\\hline
  \shortstack{$N=2(2^{n}-1)$,\\$2p,2p(p+2)$} & $2$&unknown&\cite{LUKE}&NA\\\hline
  \shortstack{$N=4(2^{n}-1)$,\\$4p,4p(p+2)$} & $4$&unknown&\cite{LUKE}&NA\\\hline
  $N=2(2^{n}-1)$ & $2$&unknown&\cite{JANG}&NA\\\hline
  $N=2p$ & $2$&unknown&\cite{KIM1}&NA\\\hline
  $N=p^{n}-1$ & $2$&unknown&\cite{KIM}&NA\\\hline
  \shortstack{$N=2m,$\\$m\equiv3 \ (\text{mod } 4)$} & $2$ &unknown&\cite{TD}&NA\\\hline
  \shortstack{$N=2p,$\\$p\equiv1 \ (\text{mod } 8)$}& $4$&$L=2p$&\cite{EDEM}&\cite{EDEM}\\\hline
  \shortstack{$N=2p,$\\$p\equiv5 \ (\text{mod } 8)$} & $\sqrt{8}$&$L=\frac{3p+1}{2}$&\cite{EDEM}&\cite{EDEM}\\\hline
  \shortstack{$N=2p,$\\$p=4y^{2}+1,$\\$y$ even} & $2$&\shortstack{(*)Many cases, e.g., $L=2p$ when \\$(i,j)\in\{(0,1),(0,3)\}$ and $2\in D_{2}^{(4,p)}$}&\shortstack{\cite{SHEN}, \\Theorem \ref{th5}}&\shortstack{\cite{ZHA} and\\ Remark \ref{re0}}\\\hline
  \shortstack{$N=2p,$\\$p=4+x^{2},$\\$x$ odd} & $2$&$L=2p$&\shortstack{\cite{SHEN}}&\shortstack{\cite{ZHA} and\\ Remark \ref{re0}}\\\hline
  $N=p^{n}-1$ & $2$&unknown&\shortstack{\cite{CHUNG} and apply Theorems \ref{le2} and \\\ref{off1} to binary Sidelnikov sequences}&NA\\\hline
  $N=p^{n}-1$ & $4$&unknown&\shortstack{\cite{CHUNG} and apply Theorems \ref{le2} and \\\ref{off1} to binary Sidelnikov sequences}&NA\\\hline
  \multicolumn{3}{l}{{\scriptsize Note: $p$ an odd prime; m,n,x,y positive integers.}}\\
  \multicolumn{3}{l}{{\scriptsize Note (*): Indices $(i,j)$ defined as in \cite[Theorem 5.11]{CUN}.}}\\
\end{tabular}
\caption{{\scriptsize Comparison of Balanced Quaternary Sequences of Even Period with Low Autocorrelation}}
\end{center}
\end{table}
\begin{theorem}\label{th6} Let $\mathbf{s}$ be a binary sequence with even period $N$ and minimal polynomial $P_{\mathbf{s}}(x)\in\mathbb{F}_{2}\left[x\right]$. If $\mathbf{u}$ is the quaternary sequence defined by $\mathbf{u}=\phi^{-1}(\mathbf{s},L^{\frac{N}{2}}(\mathbf{s}))$ then $P_{\mathbf{u}}(x)=P_{\mathbf{s}}(x)\in\mathbb{F}_{2^{2}}\left[x\right]$. If $\mathbf{u}$ is the quaternary sequence defined by $\mathbf{u}=\phi^{-1}(\mathbf{s},L^{\frac{N}{2}}(\overline{\mathbf{s}}))$ then $P_{\mathbf{u}}(x)=P_{\mathbf{s}}(x)\in\mathbb{F}_{2^{2}}\left[x\right]$ if $(x-1)^{2} \mid P_{\mathbf{s}}(x)$.
\end{theorem}

\begin{proof}
We first show the case where $\mathbf{u}=\phi^{-1}(\mathbf{s},L^{\frac{N}{2}}(\overline{\mathbf{s}}))$. Since $(x-1)^{2} \mid P_{\mathbf{s}}(x)$, by Lemma \ref{le13} we have that $\gcd(x^{N}-1,\mathbf{s}(x))=\gcd(x^{N}-1,\overline{\mathbf{s}}(x))$. Let $m(x)=\mathbf{s}(x)/\gcd(x^{n}-1,\mathbf{s}(x))$. Since $\overline{\mathbf{s}}(x)=\mathbf{s}(x)+(x^{N}-1)/(x-1)$ we have \[
\frac{\overline{\mathbf{s}}(x)}{\gcd(x^{N}-1,\overline{\mathbf{s}}(x))}=\frac{\overline{\mathbf{s}}(x)}{\gcd(x^{N}-1,\mathbf{s}(x))}
=\frac{\mathbf{s}(x)+(x^{N}-1)/(x-1)}{\gcd(x^{N}-1,\mathbf{s}(x))}=m(x)+\frac{P_{\mathbf{s}}(x)}{(x-1)}.\]

Using the notation introduced in (\ref{eqn1}), we have
\begin{eqnarray*} \gcd(x^{n}-1,\mathbf{u}(x)) & = & \gcd(x^{N}-1,\mathbf{s}(x)\mu+x^{\frac{N}{2}}\overline{\mathbf{s}}(x))\\
& = & \gcd(x^{N}-1,\gcd(x^{N}-1,\mathbf{s}(x))m(x)\mu+x^{\frac{N}{2}}\gcd(x^{N}-1,\mathbf{s}(x))(m(x)+\frac{P_{\mathbf{s}}(x)}{(x-1)}))\\
                                             & = & \gcd(x^{N}-1,\mathbf{s}(x))\gcd(P_{\mathbf{s}}(x),m(x)(\mu+x^{\frac{N}{2}})+x^{\frac{N}{2}}\frac{P_{\mathbf{s}}(x)}{(x-1)}) \\
                                             & = & \gcd(x^{N}-1,\mathbf{s}(x)),\end{eqnarray*} where the last equality holds due to the fact that $\gcd(P_{\mathbf{s}}(x),m(x))=1$, whence if $P_{\mathbf{s}}(\beta)=0$ for some $N$-th root of unity $\beta\neq1$ over $\mathbb{F}_{2^{2}}$, then we have $m(\beta)(\mu+\beta^{\frac{N}{2}})+\beta^{\frac{N}{2}}\frac{P_{\mathbf{s}}(\beta)}{(\beta-1)}=m(\beta)(\mu+1)\neq0$.
                                             \par
                                             When $\mathbf{u}=\phi^{-1}(\mathbf{s},L^{\frac{N}{2}}(\mathbf{s}))$, we have $\gcd(x^{n}-1,\mathbf{u}(x))=\gcd(x^{N}-1,\mathbf{s}(x)(x^{\frac{N}{2}}+\mu))$. The conclusion easily follows.
\end{proof}

\begin{remark}\label{re0}
.\newline\begin{enumerate}[(i)]
\item Four important, known classes of binary sequences with ideal autocorrelation are {\it Paley Class}, {\it Twin Prime Class}, {\it Hall Class}, and {\it Singer Class} \ (for a complete survey on binary sequences with ideal autocorrelation, e.g. see \cite{CAI}, \cite{GOL}); all of which can be used as base sequences to construct binary sequences of period $4l$ with optimal autocorrelation using the method shown in \cite{TD}. The complexity of these binary sequences was computed by Wang and Du in \cite{WANG} for the cases where $A$ is a cyclic shift of $B$ (where $A$ and $B$ are defined as in \cite[Theorem 5.16]{CUN}), and for each of the four classes, it was shown that $(x-1)^{2}$ is a divisor of the minimal polynomial. Then by Theorem \ref{th6} we have that if $D$ is defined as in \cite[Theorem 5.16]{CUN}, and $\mathbf{u}=\phi^{-1}(\mathbf{s},L^{\frac{N}{2}}(\overline{\mathbf{s}}))$ where $\mathbf{s}=\mathbf{s}_{\psi(D)}$ or $\overline{\mathbf{s}}_{\psi(D)}$, then $P_{\mathbf{u}}(x)=P_{\mathbf{s}}(x)$.
\item The linear complexity of the Ding-Helleseth-Martinsen binary sequences of period $2p$ can be found in \cite{ZHA}, where it is shown that $(x-1)^{2}$ is a divisor of the minimal polynomal. Then by Theorem \ref{th6} we have that if $D$ is defined as in \cite[Theorem 5.11]{CUN}, and $\mathbf{u}=\phi^{-1}(\mathbf{s},L^{\frac{N}{2}}(\overline{\mathbf{s}}))$ where $\mathbf{s}=\mathbf{s}_{\psi(D)}$ or $\overline{\mathbf{s}}_{\psi(D)}$, then $P_{\mathbf{u}}(x)=P_{\mathbf{s}}(x)$.\end{enumerate}
    \end{remark}

\section{Concluding Remarks}\label{sec5}
We have studied quaternary sequences of both even and odd period having low autocorrelation. We have constructed new families of balanced quaternary sequences of odd period and low autocorrelation using cyclotomic classes of order eight, as well as investigate the linear complexity of some known quaternary sequences of odd period. We have also constructed new families of balanced and almost balanced quaternary sequences of even period and low autocorrelation, and investigated their linear complexity as well.
\bibliographystyle{plain}

\section{Appendix}
For convenience, we denote the cyclotomic classes $D_{i}^{(4,p)}$ of order four modulo a prime $p$, simply by $D_{i}$. We will need the following lemma.
\begin{lemma}\label{le14} {\rm \cite{STO}} The five distinct cyclotomic numbers modulo $p$ of order four for odd $f$ are
\begin{eqnarray*}
(0,0) & = & (2,2) = (2,0) = \frac{p-7+2a}{16} \ (=A),    \\
(0,1) & = & (1,3) = (3,2) = \frac{p+1+2a-8b}{16} \ (=B), \\
(1,2) & = & (0,3) = (3,1) = \frac{p+1+2a+8b}{16} \ (=D), \\
(0,2) & = & \frac{p+1-6a}{16} \ (=C),                    \\
\text{all others} & = & \frac{p-3-2a}{16} \ (=E),
\end{eqnarray*} and those for even $f$ are
\begin{eqnarray*}
(0,0) & = & \frac{p-11-6a}{16} \ (=A),    \\
(0,1) & = & (1,0) = (3,3) = \frac{p-3+2a+8b}{16} \ (=B), \\
(0,2) & = & (2,0) = (2,2) = \frac{p-3+2a}{16} \ (=C), \\
(0,3) & = & (3,0) = (1,1) = \frac{p-3+2a-8b}{16} \ (=D), \\
\text{all others} & = & \frac{p+1-2a}{16} \ (=E).
\end{eqnarray*}
\end{lemma}
Here we give the proof of Lemma \ref{th67}.
\begin{proof} Let $h\in\{0,1,2,3\}\setminus \{i,j,l\}$. The balancedness comes from the simple fact that $\overline{C}_{0}\cap\overline{C}_{1}=D_{h}$, $\overline{C}_{0}\cap C_{1}=D_{l}\cup\{0\}$, $C_{0}\cap C_{1}=D_{j}$ and $C_{0}\cap\overline{C}_{1}=D_{i}$. By Lemma \ref{le1} we have \[
R_{\mathbf{u}}(\tau)=\frac{1}{2}\left[R_{\mathbf{s}_{C_{0}}}(\tau)+R_{\mathbf{s}_{C_{1}}}(\tau)\right]+\frac{\omega}{2}\left[R_{\mathbf{s}_{C_{0}},\mathbf{s}_{C_{1}}}(\tau)
+R_{\mathbf{s}_{C_{1}},\mathbf{s}_{C_{0}}}(\tau)\right].\] We show the case $(i,j,l)=(1,2,3)$. The other cases are almost identical. First notice that, by Lemma \ref{le14}, when $f$ is even resp. odd, the number $n_{0,1}(k)$ of $t$ such that $s_{C_{0}}(t)=s_{C_{1}}(t+\tau)$ for $\tau^{-1}\in D_{k}$ is \begin{equation*}\label{eqn2}\begin{cases}
4E+A+B+C+D+1,&\text{ if }k=0,\\
2(B+C+D+E),&\text{ if }k=1,\\
4E+A+B+C+D+1,&\text{ if }k=2,\\
2(B+C+D+E)+2,&\text{ if }k=3,\end{cases}\text{ resp. }\begin{cases}
2(A+B+D+E)+2,&\text{ if }k=0,\\
4E+A+B+C+D+1,&\text{ if }k=1,\\
2(A+B+D+E),&\text{ if }k=2,\\
4E+A+B+C+D+1,&\text{ if }k=3.\end{cases}\end{equation*}
The numbers $n_{1,0}(k)$, for $k=0,1,2,3$, can be calculated in the same way. When $f$ is even, we have $n_{0,1}(k)=n_{1,0}(k)$ for all $k$. When $f$ is odd, $n_{0,1}(k)=n_{1,0}(k)$ when $k=1$ or $3$, and $n_{0,1}(k)=n_{1,0}(k)-2$ if $k=0$, and $n_{0,1}(k)=n_{1,0}(k)+2$ if $k=2$. We also have, by Lemma \ref{le14}, when $f$ is even resp. odd, the number $n_{0,0}(k)$ of $t$ such that $s_{C_{0}}(t)=s_{C_{0}}(t+\tau)$ for $\tau^{-1}\in D_{k}$ is \begin{equation*}\label{eqn3}\begin{cases}
2E+3D+A+B+C+2,&\text{ if }k=0,\\
2E+3B+A+C+D+2,&\text{ if }k=1,\\
2E+3D+A+B+C,&\text{ if }k=2,\\
2E+3B+A+C+D,&\text{ if }k=3,\end{cases}\text{ resp. }\begin{cases}
4E+2(A+D)+1,&\text{ if }k=0,\\
4E+2(A+B)+1,&\text{ if }k=1,\\
4E+2(A+D)+1,&\text{ if }k=2,\\
4E+2(A+B)+1,&\text{ if }k=3.\end{cases}\end{equation*} The numbers $n_{1,0}(k)$ can be calculated in the same way. When $f$ even, we have $n_{0,0}(k)=n_{1,1}(k)$ when $k=1$ or $3$, and $n_{0,0}(k)=n_{1,1}(k)+2$ when $k=0$ or $2$. When $f$ odd, we have $n_{0,0}(k)=n_{1,1}(k)$ for all $k$. The result follows.
\end{proof}
\end{document}